\documentclass[a4paper,11pt]{amsart}
\usepackage{graphicx}
\usepackage{amsmath}
\usepackage{amssymb}
\usepackage{oldgerm}
\usepackage[all]{xy}
\newcommand{\cone}{\operatorname{cone}\nolimits}
\newcommand{\colim}{\operatorname{colim}\nolimits}
\newcommand{\add}{\operatorname{add}\nolimits}
\newcommand{\per}{\operatorname{per}\nolimits}
\newcommand{\rhom}{\operatorname{RHom}\nolimits}

\newcommand{\Hom}{\operatorname{Hom}\nolimits}

\renewcommand{\mod}{\operatorname{mod}\nolimits}

\newcommand{\Ext}{\operatorname{Ext}\nolimits}

\newcommand{\F}{\operatorname{\mathbb{F}}\nolimits}
\newcommand{\Z}{\operatorname{\mathbb{Z}}\nolimits}

\newcommand{\N}{\operatorname{\mathbb{N}}\nolimits}

\newtheorem{theo}{Theorem}[section]

\newtheorem{cor}[theo]{Corollary}

\newtheorem{lemma}[theo]{Lemma}
\newtheorem{prop}[theo]{Proposition}
\newtheorem{defi}[theo]{Definition}
\newtheorem{rem}[theo]{Remark}
\newtheorem{example}[theo]{Example}

\newcommand{\cc}{\mathcal{C}}
\newcommand{\iso}{\stackrel{_\sim}{\rightarrow}}
\newcommand{\eps}{\varepsilon}

\newcommand{\ko}{\;\; ,}
\newcommand{\ten}{\otimes}

\begin{document}
\baselineskip=15pt
\title[The integral cluster category]{The integral cluster category}
\author{Bernhard Keller and Sarah Scherotzke}
\address{B.~K.~: Universit\'e Paris Diderot - Paris~7, Institut
Universitaire de France, UFR de Math\'ematiques, Institut de
Math\'ematiques de Jussieu, UMR 7586 du CNRS, Case 7012, B\^atiment
Chevaleret, 75205 Paris Cedex 13, France}
\address{S.~S.~: Universit\'e Paris Diderot - Paris~7, UFR de Math\'ematiques, Institut de
Math\'ematiques de Jussieu, UMR 7586 du CNRS, Case 7012, B\^atiment
Chevaleret, 75205 Paris Cedex 13, France}

\email{keller@math.jussieu.fr, scherotzke@math.jussieu.fr}

\subjclass[2000]{}

\begin{abstract} Integral cluster categories of acyclic quivers have recently been used
in the representation-theoretic approach to quantum cluster algebras.
We show that over a principal ideal domain, such categories behave
much better than one would expect: They can
be described as orbit categories, their indecomposable
rigid objects do not depend on the ground ring and the
mutation operation is transitive.
\end{abstract}

\maketitle

\section{introduction}

Cluster categories were introduced in
\cite{BuanMarshReinekeReitenTodorov06} for acyclic quivers and
independently in \cite{CalderoChapotonSchiffler06} for Dynkin
quivers of type $A$. They have played an important r\^ole in the
study of Fomin-Zelevinsky's cluster algebras
\cite{FominZelevinsky02}, cf.~the surveys \cite{BuanMarsh06}
\cite{Keller09b} \cite{Keller10b} \cite{Reiten06} \cite{Ringel07}.

Integral cluster categories appear naturally in the study of
quantum cluster algebras as defined and studied in
\cite{BerensteinZelevinsky05} and \cite{FockGoncharov09}. Indeed,
one would like to interpret the quantum parameter $q$ as the
cardinality of a finite field \cite{Rupel10a} and in order to
study the cluster categories over all prime fields simultaneously
\cite{Qin10}, one considers the cluster category over the ring of
integers, cf.~the appendix to \cite{Qin10}. In this paper, we
continue the study begun there: For an acyclic quiver $Q$ and a
principal ideal domain $R$, we construct the cluster category
$\cc_{RQ}$ using Amiot's method \cite{Amiot09} as a triangle
quotient of the perfect derived category of the Ginzburg dg
algebra \cite{Ginzburg06} associated with the path algebra $RQ$.
On the other hand, we define the category $\cc^{orb}_{RQ}$ as the
category of orbits of the bounded derived category of $RQ$ under
the action of the autoequivalence $\Sigma^{-2}S$, where $S$ is the
Serre functor and $\Sigma$ the suspension functor. In the case
where $R$ is a field, cluster categories were originally defined as
$\cc^{orb}_{RQ}$ in \cite{BuanMarshReinekeReitenTodorov06}
and it was shown in \cite{Amiot09} that the two
definitions are equivalent. Our first main result is the existence
of a natural equivalence for any principal ideal domain $R$
\[
\cc^{orb}_{RQ} \iso \cc_{RQ}.
\]
This shows in particular that the orbit category is triangulated.
For the case where $R$ is a field, this has been known since
\cite{Keller05}; in the general case, it is quite surprising since
the algebra $RQ$ is of global dimension $\leq 2$ and the
proof in \cite{Keller05} strongly used the fact that for
a field $\mathbb{F}$, the path algebra $\mathbb{F}Q$
is of global dimension $\leq 1$.

Our second main result states that all indecomposable
rigid objects of $\cc_{RQ}$ are either images of rigid
indecomposable $RQ$-modules or direct factors of
the image of $\Sigma RQ$. If we combine this with
Crawley-Boevey's classification \cite{CrawleyBoevey96}
of rigid indecomposable $RQ$-modules, we obtain that
the classification of the rigid indecomposable objects
of $\cc_{RQ}$ is independent of the principal ideal
domain $R$ and that iterated
mutation starting from $RQ$ reaches all indecomposable
rigid objects. Again, this is well-known in the
field case, cf.
\cite{BuanMarshReinekeReitenTodorov06}
\cite{HappelUnger05a}
\cite{Hubery10}, but quite surprising in the general case.

The paper is structured as follows:
In the second section, we recall general adjointness
relations between the derived tensor and the derived
$\Hom$-functor for dg algebras over any commutative
ring $R$. We define the (relative) Serre functor.

In the third section, we consider the derived category of an
$R$-algebra $A$ (finitely generated projective over $R$) and an
endofunctor $F$ of the derived category of $A$ which is given by
the derived tensor product with an $A$-bimodule complex $\Theta$.
The tensor dg algebra associated with $\Theta$ is a differential
graded algebra which we denote by $\Gamma$. We give sufficient
conditions for the orbit category $\mathcal{C}^{orb}$ of the
perfect derived category $\per(A)$ by $F$ to embed canonically
into the triangulated quotient category $\per(\Gamma)/D_{\per
(R)}(\Gamma)$. Here $D_{\per (R)}(\Gamma)$ denotes the derived
category of the differential graded $\Gamma$-modules whose restrictions
to $R$ are perfect complexes. The methods used in this section
generalize the approach used in \cite{Keller11b} to algebras over
arbitrary commutative rings.

In the fourth section, we consider the orbit category of the
perfect derived category $\per(RQ)$ under the action of the
auto-equivalence given by $\Sigma^{-2} S$. This functor is given
by the total derived functor associated with the $RQ$-bimodule
complex $ \Theta=\Sigma^{-2} \Hom_{R}(RQ, R) $. By a result of
\cite{Keller11b}, the differential graded tensor algebra of
$\Theta$ is isomorphic to the Ginzburg algebra $\Gamma$ associated
to the quiver $Q$ with the zero potential. So using the results of
the third section, we can embed the orbit category into the
integral cluster category $\per(\Gamma) / D_{\per (R)}(\Gamma)$.
Furthermore, we show in this section that a relative
$3$-Calabi-Yau property holds for $D(\Gamma)$.

In the fifth section, under the assumption that $R$ is a principal
ideal domain, we show that all rigid indecomposable objects in the
integral cluster category come from modules and their suspensions
and that the embedding given in section~3 is an equivalence of
categories. Hence the orbit category  $\mathcal{C}^{orb}$ is
triangulated and the integral cluster category is relative
$2$-Calabi-Yau. Using a result by Crawley-Boevey
\cite{CrawleyBoevey96} we establish a bijection between the rigid
indecomposable objects in the cluster category over a field $\F$
and those over a ring induced by the triangle functor
$?\otimes^L_R \F$.

In the last section, we show using the bijection between rigid
objects in $\cc_{RQ}$ and $\cc_{\F Q} $, that all cluster-tilting
objects are related by mutations.

\section*{Acknowledgment}
The second-named author thanks the `Fondation Sciences
Math\'ematiques de Paris' for a postdoctoral fellowship during
which this project was carried out. Both authors are grateful to
the referee for many helpful comments.

\section{Derived categories over commutative rings}

Let $R$ be a commutative ring. For an associative differential
graded $R$-algebra $A$ which is cofibrant as an $R$-module (cf.
section~2.12 of \cite{KellerYang11}), we denote by $D(A)$ the {\em
derived category} of dg $A$-modules, by $\per(A)$ the {\em perfect
derived category}, i.e. the thick subcategory of $D(A)$ generated
by $A$, and by $D_{\per (R)}(A)$ the full subcategory of $D(A)$
whose objects are the dg $A$-modules whose underlying complex of
$R$-modules is perfect. Throughout this article, we denote by
$\Sigma $ the shift functor in the derived category. If the
underlying $R$-module of $A$ is finitely generated projective over
$R$, we denote by $ S_{R}$ the (relative) Serre functor of $D(A)$
given by the total derived functor of tensoring with the
$A$-bimodule $\Hom^._{R}(A, R)$. Here, for two dg $A$-modules $L$
and $M$, we denote by $\Hom^._A(L,M)$ the dg $R$-module whose
$n$th component is the $R$-module of morphisms of graded
$A$-modules $f: L\to M$ homogeneous of degree $n$ and whose
differential sends such an $f$ to $d_M \circ f - (-1)^n f \circ
d_L$. We denote by $\rhom$ the total derived functor of $\Hom^.$.
We define $A^e$ to be the dg algebra $A \otimes_R A^{op}$. The
following well-known isomorphisms will often be used in the rest
of the article.

\begin{lemma}\label{basic}
Let $A$ and $B$ be two dg $R$-algebras which are cofibrant over $R$.

(1) Let $M \in D(B\otimes A^{op})$, $L \in D(A)$ and $N \in D(B)$.
There is a bifunctorial isomorphism
\[
\rhom_B( L\otimes_A^L M,N) \iso \rhom_A( L , \rhom_B(M,N))
\]
in $D(R)$.

(2) For $P \in \per (B)$ and $M\in D(B)$, there is a bifunctorial
isomorphism
\[
M \otimes^L_B \rhom_B(P,B) \iso \rhom_B(P,M)
\]
in $D(B^{op})$.

(3) For all $L$ and $M$ in $D(A)$, there is a bifunctorial isomorphism
\[
\rhom_{A^e} (A, \rhom_R(L,M))\iso \rhom_A(L,M)
\]
in $D(A^{op})$ .

\end{lemma}
\begin{proof}
We denote by $pM$ a cofibrant replacement of $M$ and by $iM$ a
fibrant replacement of $M$. We have $L \otimes^L_A M \cong pL
\otimes_A p M$. Therefore, we have
\begin{align*}
 \rhom_B(L \otimes_A^L M, N) &=\Hom_B^.(pL \otimes_A pM, N)\\
 &\cong \Hom^._A(pL, \Hom^._B(pM, N))\\
 &=\rhom_A( L , \rhom_B(M,N)).
\end{align*}
This proves (1).

For part (2), we observe that we have a bifunctorial morphism from
$pM\otimes_B\Hom^._B(pP,B)$ to $\Hom^._B(pP,pM)$, which is
invertible in $D(R)$ for $P=A$ hence for all objects in $\per (A)$.

For part (3), we have
$$
\rhom_{A^e} (A, \rhom_R(L,M))=\Hom^._{A^e}(pA, \Hom^._R(pL, iM))\ko
$$
where $pA$ is a cofibrant as a dg $A^e$-module. Then $pA$ is also
cofibrant as a dg $A$-module and as a dg $A^{op}$-module. We have
$\Hom^._{A^e}(pA, \Hom^._R(pL, iM)
)\subset \Hom^._{A^{op}}(pA, \Hom^._R(pL, iM))$ and by $(1)$, there
is a bifunctorial isomorphism between $\Hom^._{A^{op}}(pA,
\Hom^._R(pL, iM))$ and $\Hom^._R(pA \otimes_{A^{op}} pL, iM)$.
This isomorphism induces a bijection between $\Hom^._{A^e}(pA,
\Hom^._R(pL, iM) )$, which consists of all the elements of
$\Hom^._{A^{op}}(pA, \Hom^._R(pL, iM)) $ that commute with the right
action of $A$, and $\Hom^._A( pA\otimes_{A^{op}} pL, iM)$. Now
$\Hom^._A( pA\otimes_A pL, iM)$ is isomorphic to $\rhom_A( L, M) $,
which finishes the proof.
\end{proof}

\begin{prop}\label{calabi} Suppose that the underlying $R$-module of $A$
is finitely generated projective.
For $L \in D(A)$ and $M \in \per (A)$, we have the
following canonical bifunctorial isomorphism
\[
\rhom_R(\rhom_{A}(M, S_{R} L),R) \iso \rhom_{A}(L,M).
\]
If $\Hom_R(A,R)$ belongs to $\per (A)$, then $S_R$ is an
auto-equivalence of $\per (A)$ with inverse $\rhom_A( \rhom(?,R), A)$.
\end{prop}
\begin{proof} By applying part (2) of \ref{basic} twice
we obtain
$$\rhom_A(M, L\otimes^L_A \rhom_R(A,R)) \cong L \otimes^L_A
\rhom_A(M, \rhom_R(A,R)).$$ By part (1) of \ref{basic}, we obtain
$$\rhom_A(M, \rhom_R(A,R))  \cong \rhom_R(M, R). $$
Therefore, we have
\begin{align*}
\rhom_R(\rhom_A(M, &L \otimes^L_A \rhom_R(A,R)) ),R) \\ &\cong \rhom_R(L \otimes^L_A \rhom_R(M,R),R) \\
&\cong \rhom_A(L,\rhom_R(\rhom_R(M,R),R)) \\ &\cong \rhom_A(L,M)
\end{align*} by (1) of \ref{basic}. This proves the first part.

If we choose $M=A$, we get, by the first statement, $$S_R L \cong
\rhom_{R}(\rhom_{A}(L,A), R) .$$ Now we use the fact that
$\rhom_R(?,R)$ and $\rhom_A(?,A)$ are duality functors on $\per (A)$.

\end{proof}

\section{Embeddings of orbit categories}

Let $A$ be an associative $R$-algebra which is finitely generated
projective as an $R$-module and let $\Theta$ be a complex of
$A$-$A$-bimodules. We suppose that $\Theta$ is cofibrant as a dg
$A$-bimodule and that $\Theta$ is perfect as a dg $R$-module. We
denote by $F:D(A) \to D(A)$ the functor $?\otimes_A^L \Theta$. We
define the dg algebra $ \Gamma=T_A(\Theta)$ to be the tensor algebra
over $A$ given by
\[
A\oplus \Theta \oplus (\Theta\otimes_A \Theta)
\oplus \cdots \oplus (\Theta\otimes_A \cdots \otimes_A \Theta)
\oplus \cdots \quad .
\]
Then $ \Gamma$ is homologically smooth over $R$ by
\cite[3.7]{Keller11b}. For $N \ge 0$, we denote by $ \Gamma^{>N}$
the ideal $\bigoplus_{n>N} \Theta^{\otimes_A n}$ of $ \Gamma$ and
put $ \Gamma^{\le N} = \Gamma/ \Gamma^{> N}$.  Then $D_{\per (R)}(
\Gamma)$ is contained in $\per ( \Gamma)$ and $ \Gamma^{\le N }$
lies in $D_{\per (R)}( \Gamma)$ for all $N \in \N$. We consider the
category $\mathcal{C}( \Gamma)= \per( \Gamma)/D_{\per (R)}( \Gamma)$
and compute its morphism spaces. We have a functor $?\otimes^L_A
\Gamma: \per (A) \to \mathcal{C}( \Gamma)$ and a restriction functor
$\per( \Gamma) \to D(A)$ induced by the natural embedding of $A$
into $ \Gamma$. For any $Y \in \per( \Gamma)$ and any $N\in \N$, let
$F^N(Y)=Y \ten_A \Theta^{\ten_A N}$ and let $m_N: F^N(Y) \to Y$ be
induced by the multiplication.

We assume that for any $X \in D_{\per (R)}(A) $ there is an $n \in \N $ such that $\Hom_{D(A)}(F^n(A),X)$ vanishes.

\begin{lemma} For $Y$ in $\per( \Gamma)$, we have the following isomorphisms
\begin{align*}
\colim_N \Hom_{D( \Gamma)}( \Gamma^{>N}, Y) &\cong
\Hom_{\mathcal{C}( \Gamma)}( \Gamma, Y) \\ &\cong \colim_N  \Hom_{D(A)}
(F^{N+1}(A), Y|_A).
\end{align*}
\end{lemma}

\begin{proof}
By definition, the space $\Hom_{\mathcal{C}( \Gamma)}( \Gamma, Y)$
is given by
\[
\colim_{M_{\Gamma}}  \Hom_{D( \Gamma)}(\Gamma',Y)\ko
\]
where $M_{\Gamma}$ denotes the category of all morphisms $s: \Gamma'
\to \Gamma$ in $D(\Gamma)$ such that $\cone(s) $ lies in $D_{\per
R}(\Gamma)$. We consider the exact sequence
\[
\xymatrix{0 \ar[r] & \Gamma^{>N} \ar[r]^{e_N} & \Gamma  \ar[r] &  \Gamma^{\le N} \ar[r] &  0.}
\]
As
$\Gamma^{\le N} $ vanishes in $\mathcal{C}(\Gamma)$, the embedding
$e_N$ is an isomorphism for any $N$ in $\N$.
The transition maps in the direct system are induced by
the embedding of $\Gamma^{>N+1}$ into $\Gamma^{>N}$ and the maps from
$\Hom_{D(\Gamma)}(\Gamma^{>N}, Y)$ to $\Hom_{\mathcal{C}(\Gamma)}(\Gamma, Y)$ by
composing with the inverse of $e_N$.

By a classical result of Verdier, it is sufficient to show that for
every morphism $s$ in $M_ \Gamma$, there is an $N\in \N$ such that
$e_N$ factors through $s$. It is therefore sufficient to show that
for every $Y \in D_{\per (R)}( \Gamma)$ there is an $N\in \N$ such
that the space $\Hom_{D( \Gamma)}( \Gamma^{>N} , Y)$ vanishes. We
have $ \Gamma^{>N} =\Theta^{\otimes_A (N+1)}\otimes_A  \Gamma$ and
by adjunction
\begin{align*}
\Hom_{D( \Gamma)}(\Theta^{\otimes_A (N+1)}\otimes_A  \Gamma, Y) &\cong
\Hom_{D(A)} (\Theta^{\otimes_A (N+1)}, \rhom_{ \Gamma}( \Gamma, Y))\\
&\cong
\Hom_{D(A)} (F^{N+1}(A), Y|_A ).
\end{align*}
The transition maps of the direct system $\colim_N \Hom_{D(A)} (F^{N+1}(A), Y|_A )$
are given by applying $F$ and composing with the multiplication map
$m_1:Y\otimes^L_A \Theta \to Y$. If $Y \in D_{\per (R)}( \Gamma)$,
then $Y|_A \in D_{\per (R)}(A)$ and by the assumption there is an
$N\in \N $ such that $\Hom_{D(A)}(F^N(A),Y|_A ) $ vanishes. By the
above isomorphism, any map from $\Gamma^{>N}$ to $Y$ vanishes.
Therefore, the colimit
\[
\colim_N \Hom_{D( \Gamma)}( \Gamma^{>N},Y)
\cong \colim_N \Hom_{D(A)} (F^{N+1}(A), Y|_A)
\]
computes $\Hom_{\mathcal{C}( \Gamma)}( \Gamma,Y)$.
\end{proof}
\begin{defi}
Let $F: \mathcal{C} \to \mathcal{C}$ be an endofunctor of an additive
category $\mathcal{C}$. The {\em orbit category} $\mathcal{C} /
\langle F \rangle $ of $\mathcal{C}$ by $F$ is the category with
the same objects as $\mathcal{C} $ and the spaces of morphisms
\[
\Hom_{\mathcal{C} / \langle F \rangle }(M,N)=\colim_{l\in \N}
\bigoplus_{i\in \N} \Hom_{\mathcal{C}} ( F^i(M), F^l(N)).
\]
\end{defi}

\begin{theo}\label{embedding} Let $Y=Y_0 \otimes^L_A  \Gamma$ for
some object $Y_0$ of $\per(A)$ and suppose that $\Theta$ belongs to
$\per (A)$. Then we have
\[
\Hom_{\mathcal{C}( \Gamma)}( \Gamma, Y)\cong
\colim_N\bigoplus_{l\in \N} \Hom_{D^b(A)} (F^N(A), F^l (Y_0))
\] and $?\otimes^L_A  \Gamma$
induces a fully faithful embedding of the orbit category
$\mathcal{C}^{orb} $ of $\per (A)$ by $F$ into $\mathcal{C}( \Gamma)$.
Furthermore $\mathcal{C}( \Gamma)$ equals its smallest thick
subcategory containing the orbit category.
\end{theo}
\begin{proof}
By the previous lemma, we have $$\Hom_{\mathcal{C}( \Gamma)}(
\Gamma, Y)\cong \colim_N  \Hom_{D(A)} (F^{N+1}(A), Y|_A).$$ But
$$Y|_A =(Y_0\otimes^L_A  \Gamma)|_A \cong \bigoplus_{l \in \N }Y_0
\otimes^L_A \Theta^{\otimes_A l}\cong \bigoplus_{l \in \N}
F^l(Y_0).$$ This proves the first statement because $F^{N+1}(A)$ is
perfect in $D(A)$. Using the fact that $\Hom_{\mathcal{C}(
\Gamma)}(?\otimes^L_A  \Gamma, Y)$ and
\[
\colim_N\bigoplus_{l\in \N} \Hom_{D(A)} (F^N(?), F^l(Y_0))
\]
are homological functors on $\per (A)$, we obtain that
$$\Hom_{\mathcal{C}( \Gamma)}(L\otimes^L_A  \Gamma, Y)\cong \colim_N
\bigoplus_{l\in \N} \Hom_{D(A)} (F^N(L), F^l (Y_0))$$ for all $L\in
\per (A)$. Since we have $\colim_N\bigoplus_{l\in \N} \Hom_{D(A)}
(F^N(L), F^l (Y_0))=\Hom_{\mathcal{C}^{orb}}(L, Y_0)$, the functor
$?\otimes_A^L  \Gamma$ induces a fully faithful embedding of the
orbit category into $\mathcal{C}( \Gamma)$. The functor
$?\otimes^L_A  \Gamma$ induces a triangle functor from $\per (A)$ to
$\per ( \Gamma)$ such that $A$ maps to $ \Gamma$. The triangle
closure of the image of $\per (A)$ is therefore $\per  (\Gamma)$. The
last statement now follows from the commutativity of the following
diagram
\[
\xymatrix{ \per (A) \ar[r]^{?\otimes^L_A  \Gamma} \ar[d] & \per (\Gamma
)\ar[d]
\\ \mathcal{C}^{orb}  \ar[r]_{? \otimes^L_A  \Gamma} & \mathcal{C}( \Gamma) .}
\]
\end{proof}

\begin{rem} \label{c-orbit-morphisms}
Note that if $F$ is an equivalence, then the colimit in \ref{embedding} is given by $$ \bigoplus_{l\in \Z}
\Hom_{D(A)} (A, F^l (Y_0)).$$
\end{rem}

\smallskip

Suppose that $F=? \otimes^L_A \Theta$ sends $\per (A)$ to itself. Let
$\F$ be a field  and $ \pi: R \to \F$ a ring homomorphism. We denote
by $\F A$ the scalar extension $A\ten_R \F$, by $\mathcal{C}(\F \!
\Gamma)$ the category $\per  (\F \! \Gamma) / D^b(\F \! \Gamma) $, by
$\F \! F$ the functor $? \otimes^L_{\F \! A} (\F \otimes^L \Theta)$
on $ \per (\F \! A )$ and by $\mathcal{C}_{\F }^{orb}$ the orbit
category of $\per (\F \! A)$ by $\F \! F$.

\begin{cor}\label{scalar extension}
The following diagram commutes:
\[
\xymatrix{ \per (A) \ar[r] \ar[d]_{? \otimes^L_R \F} &
\mathcal{C}^{orb} \ar[d] \ar[r]^{?\otimes_A^L  \Gamma} & \mathcal{C}
( \Gamma) \ar[d] & \per ( \Gamma) \ar[l] \ar[d]^{ ?\otimes^L_R \F}
\\ \per( \F \! A)  \ar[r] & \mathcal{C}_{\F}^{orb}
\ar[r]_-{?\otimes_{\F \! A} ^L \F \! \Gamma} & \mathcal{C}(\F \!
\Gamma) & \per (\F \! \Gamma) \ar[l]  }\]
\end{cor}
\begin{proof}
We get the functor from $\mathcal{C}( \Gamma)$ to $\mathcal{C}(\F \!
\Gamma)$ from the fact that  $?\otimes^L_R \F $ maps $D_{\per (R)}(
\Gamma) $ into $ D^b( \F \! \Gamma)$. Since $A$ is cofibrant over
$R$, every complex that is cofibrant over $A$ or $A^e$ is also
cofibrant over $R$. Therefore $\F \otimes_R \Theta$ is cofibrant as
a dg $\F \! A^e$-module and the following diagram commutes
\[
\xymatrix{ \per (A) \ar[r]^F \ar[d]_{?\otimes^L_R \F } &\per (A) \ar[d]^{?\otimes_R^L \F }\\
\per (\F \! A) \ar[r]_{ \F \! F} &\per (\F \! A).}
\]
This proves the
existence of a natural functor from $\mathcal{C}^{orb} $ to
$\mathcal{C}_{\F}^{orb}$. The commutativity of the middle square
follows from the diagram in the proof of \ref{embedding}.

\end{proof}
\section{The integral cluster category}
Let $Q$ be a finite quiver without oriented cycles. An {\em
$RQ$-lattice} is a finitely generated $R Q$-module which is free
over $R$. We denote by $D(R Q)$ the derived category of $R
Q$-modules. We denote by $\Gamma$ the Ginzburg dg algebra
\cite{Ginzburg06} associated to $(Q,0)$ and by $D(\Gamma)$ the
derived category of differential graded $\Gamma$-modules. We refer
to \cite[2.12]{KellerYang11} for an introduction to the Ginzburg
algebra and its derived category. The {\em integral cluster
category} is defined as the triangle quotient
\[
\mathcal{C}_{R Q}= \per (\Gamma)/ D_{\per (R)}( \Gamma).
\]
In analogy with \cite{BuanMarshReinekeReitenTodorov06}, we define
$\mathcal{C}^{orb}$ to be the orbit category of $\per(RQ)$
by the auto-equivalence $S_{R}\Sigma^{-2}$.

\begin{theo}\label{faithful}
The functor $? \otimes_{RQ}^L\Gamma$ induces a fully faithful
embedding of $\mathcal{C}^{orb}$ into $ \mathcal{C}_{R Q}$. The
category $\mathcal{C}_{R Q}$ is the triangulated hull of
$\mathcal{C}^{orb}$.
\end{theo}
\begin{proof}
Let $\Theta=\Sigma^{-2} \rhom_{R Q^e}(R Q, {R Q}^e)\cong \Sigma^{-2}
\Hom_{R}(R Q, R)$. The tensor functor $?\otimes^L_A \Theta $ induces
the functor $S_R \Sigma^{-2}$ and restricts to an equivalence of
$\per (R Q)$ by \ref{calabi} . By \cite[6.3]{Keller11b}, the tensor
algebra $T_A(\Theta)$ is quasi-isomorphic to the Ginzburg algebra
$\Gamma$. Now Theorem~\ref{embedding} yields the statement.
\end{proof}

\begin{prop}\label{derivedcalabi}
The category $D(\Gamma)$ satisfies the relative 3-Calabi-Yau
property, i.e. for $Y\in D(\Gamma)$ and $ X\in D_{\per{R}}(\Gamma)$,
there is a bifunctorial isomorphism
\[
\rhom_{R}(\rhom_{D( \Gamma)}( X,Y),R)\cong \rhom_{D( \Gamma)}(Y,
\Sigma^3 X)
\]
for any $Y\in D(\Gamma)$ and $ X\in D_{\per{R}}(\Gamma)$.
\end{prop}
\begin{proof} By \cite[4.8]{Keller11b}, the dg module
$\Omega= \rhom_{\Gamma^e}(\Gamma, \Gamma^e)$ is isomorphic
to $\Sigma^{-3}\Gamma$ in $D(\Gamma^e)$. Now using Lemma~\ref{basic}
we obtain the required isomorphism by the same proof as in
\cite[4.8]{Keller11b}.
\end{proof}

\begin{cor} \label{3calabi}
Suppose that the ring $R$ is hereditary. Let $L$ be an object in
$D(\Gamma)$ and let $Z$ be an object in the subcategory
$D_{\per{R}}(\Gamma)$. Then we have
\[
\Hom_{R}(\Hom_{D(\Gamma) }(Z, L), R) \oplus \Ext^1(\Hom_{D( \Gamma)
}(Z, \Sigma ^{-1} L), R) \cong \Hom_{D( \Gamma) }(L, \Sigma^3 Z).
\]
\end{cor}
\begin{proof}
As $R$ is hereditary, every object $X$ in $D(R)$ is isomorphic to
the sum of its shifted homologies $\bigoplus_{n\in \Z} \Sigma^{-n}
H^n (X)$. As $H^n \rhom_{D( \Gamma)}(Z, L)= \Hom_{D( \Gamma)}(Z,
\Sigma^n L)$, we have $\rhom_{D( \Gamma)}(Z, L)\cong
\bigoplus_{n\in \Z} \Sigma^{-n} \Hom_{D( \Gamma)}(Z, \Sigma^n L)$.
Therefore $\rhom_{R}(\rhom_{D(\Gamma)}(Z,L),R)$ is isomorphic to
\[
\prod_{n\in \N} \Sigma^{n} \rhom_{R}( \Hom_{D( \Gamma)}(Z,
\Sigma^n L), R).
\]
Furthermore, we have $\rhom_{R}(M,R)\cong
\Hom_{R}( M, R)\oplus \Sigma^{-1}  \Ext^1_{R}(M,R) $ for any $R
$-module $M$. Therefore, the homology of
$\rhom_{R}(\rhom_{D(\Gamma)}(Z,L),R)$ in degree zero is given by
\[
\Hom_{R}(\Hom_{D(\Gamma) }(Z, L), R) \oplus
\Ext^1(\Hom_{D(\Gamma) }(Z, \Sigma ^{-1} L), R).
\]
We obtain the statement by comparing the homology in degree zero
in \ref{derivedcalabi}  and using the fact that
the homology of $\rhom_{D(\Gamma)}(L, \Sigma^3 Z)$ in degree zero
is given by $\Hom_{D( \Gamma) }(L, \Sigma^3 Z)$.
\end{proof}

\section{rigid objects}
We assume from now on that $R$ is a hereditary ring. Our goal in
this section is to show that, if $R$ is a principal ideal domain,
each rigid object of the integral cluster category is either the
image of a rigid indecomposable $RQ$-module or the suspension of the
image of an indecomposable projective $RQ$-module. We also show that
the orbit category $\mathcal{C}^{orb}$ and the integral cluster
category are equivalent, so that the orbit category is triangulated.

By \cite[Theorem 1]{CrawleyBoevey96} all rigid indecomposable $R
Q$-modules are lattices and there is a bijection between the rigid
indecomposable lattices and the real Schur roots of the quiver $Q$
given by the rank vector. Following \cite{Amiot09}, we define the
{\em fundamental domain} to be the $R$-linear subcategory
\[
\mathcal{F}= D_{\le 0} \cap { }^{\perp}\! D_{\le -2} \cap \per (
\Gamma)
\]
of $\per (\Gamma)$, where $ { }^{\perp}\! D_{\le n}$ denotes the full
subcategory whose objects are the $X\in D(\Gamma)$ such that
$\Hom_{D(\Gamma)}(X,Y)$ vanishes for all $Y\in D_{\le n}$.  Let $\pi:\per (\Gamma )\to
\mathcal{C}_{R Q} $ be the canonical triangle functor.
\begin{theo}\label{fundamental}
For every object $Z$ of $\mathcal{C}_{R Q}$, there is an $N \in \Z$ and an
object $Y\in \mathcal{F}[N]$ such that $\pi(Y)$ is isomorphic to
$Z$.
\end{theo}
\begin{proof}
For every object $X \in \per(\Gamma)$, there is an $N \in \Z $ and
an $M\in \Z$ such that $X \in D_{\le N} $ and $X \in { }^{\perp}\! D_{\le
M}$. This follows from the facts that $\Gamma \in  D_{\le
0} ( \Gamma)\cap { }^{\perp}\! D_{\le -1}( \Gamma)$ and that the property is
stable under taking shifts, extensions and direct factors. So let $X \in { }^{\perp}\!  D_{\le
N-2}$ for some $N \in \Z$. Consider the canonical triangle
$$\tau_{\le N}(X) \to X \to \tau_{>N} (X) \to \Sigma \tau_{\le
N}(X).$$ As $\tau_{> N}(X) \in D_{\per (R)}( \Gamma) $ and $\pi$ is
a triangle functor, the objects $\pi( \tau_{\le N}(X))$ and $ \pi
(X)$ are isomorphic. It remains to show that $\tau_{\le N}(X)\in
{ }^{\perp}\! D_{\le N-2}$, which is equivalent to the fact that
$\Sigma^{-1} \tau_{> N}(X)$ lies in ${ }^{\perp}\! D_{\le N-2}$. By
\ref{3calabi}, for each object $L$ of $D(\Gamma)$, the sum
\[
\Hom_{R}(\Hom_{D( \Gamma) }(\Sigma^{-1} \tau_{>
N}(X), L), R) \oplus \Ext_{R}^1(\Hom_{D( \Gamma) }(\Sigma^{-1}
\tau_{> N}(X), \Sigma ^{-1} L), R)
\]
is isomorphic to $\Hom_{D(\Gamma) }(L, \Sigma^{2} \tau_{> N}(X))$.
This group vanishes for all $L \in D_{\le N-2}$. We fix an object
$L$ in $D_{\le N-2}$. The $R$-module
\[
H=\Hom_{D( \Gamma)}(\Sigma^{-1} \tau_{> N}(X), L)
\]
is left orthogonal to $R$ and so
has to be a torsion module. Since $\Sigma L$ also lies in $D_{\leq
N-2}$, the group $\Ext^1_R(H,R)$ vanishes and so we have $H=0$.
Therefore, the object $\tau_{\le N}(X)$ lies in $ \mathcal{F}[N]$
and its image is isomorphic to the image of $X$ in the integral
cluster category.
\end{proof}

Following \cite{Amiot09}, we define an {\em
$\add(\Gamma)$-resolution} of an object $M \in \per (\Gamma)$ to be a
triangle
$$ P_0 \to P_1 \to M \to \Sigma P_0$$ with $P_0, P_1$ in $\add (\Gamma)$.
\begin{lemma}\label{add}
An object $X\in \per (\Gamma)$ has an $\add(\Gamma)$-resolution if and only if $X$ lies in $\mathcal{F}$.
\end{lemma}
\begin{proof}
If $X$ lies in $\mathcal{F}$, then $X$ has an
$\add(\Gamma)$-resolution by the proof of \cite[2.10]{Amiot09}. Now
let $$P_1 \to P_0 \to X \to \Sigma P_0$$ be an
$\add(\Gamma)$-resolution.  Applying the homology functor to this
triangle, we get a long exact sequence. Using the fact that $\Gamma$
has non-zero homology only in non-positive degree, we see that $X$
lies in $D_{\le 0}$. Next we apply the functor $\Hom_{D(\Gamma)}( ?,
Y)$ to the $\add(\Gamma)$-resolution for any object $Y \in D_{\le
-2}.$ Then we obtain a long exact sequence
\[
\cdots \to \Hom(\Sigma P_0, Y) \to \Hom(X,Y) \to
\Hom(P_0, Y) \to \Hom(P_1,Y) \to \cdots\quad .
\]
All terms in this sequence
have to vanish and therefore $X$ belongs to  ${ }^{\perp}\!D_{\le -2}$.
\end{proof}

We have $$\Hom_{\mathcal{C}_{R Q} }(\Gamma, \Gamma) \cong
\Hom_{\mathcal{C}^{orb}}(R Q, R Q)\cong \Hom_{\mathcal{C}^{orb}}(R
Q, R Q)\cong R Q$$ by \ref{faithful}. Therefore we have a functor
$G=\Hom_{\mathcal{C}_{R Q} }(\Gamma, ?)$ from $\mathcal{C}_{R Q } $
to the category of $R Q$-modules. Note that $G$ vanishes on $\Sigma
\Gamma$, hence $G$ factors through the quotient category
$\mathcal{C}_{R Q}/ \add (\Sigma \Gamma)$.

\begin{example}
We consider the quiver
$Q:\xymatrix{ 1 \ar[r]^\alpha & 2}$ and $R= \Z$. Let $M$ be the module
given by the quiver representation $ 0\leftarrow \Z / 2 \Z$.
Then $M$ has the projective resolution
\[
\xymatrix{0 \ar[r] & P_1 \ar[r]^-{[2\; \alpha]^t} & P_1 \oplus P_2 \ar[r]^-{[\alpha \; -2]} & P_2 \ar[r] & M \ar[r] & 0}
\]
By applying the functor $\Hom_{\Z Q}(?,M)$ to the resolution we
see that the group of selfextensions of $M$ is isomorphic to $\Z/2\Z$
so that $M$ is not rigid. Let $M'$ be the image of $M$ in $\per(\Gamma)$.
We have $G(M') \cong M$ but $M'$ does not lie in the fundamental domain as
$\Hom_{D(RQ)}(M, \Sigma^{2} P_1)$ is isomorphic to $\Z/2\Z$.
But clearly $M'$ belongs to ${}^\perp \! D_{\le -3}(\Gamma)$, hence by the
proof of \ref{fundamental}, we have that
$\tau_{\le -1}( M') \cong M'$
in $\mathcal{C}_{\Z Q}$ and
$\Sigma^{-1} \tau_{\le -1}( M') \in \mathcal  {F}$.


\end{example}

\begin{lemma}\label{G}
Let $M$ be an object in $\mathcal{C}_{R Q}$. Then $G(GM\otimes^L_{R
Q} \Gamma) $ and $GM$ are isomorphic in $\mathcal{C}_{R Q}$. If $G
M$ is a lattice, then $G(GM \otimes^L_{R Q}\Gamma)$ viewed as an
element of $\per (\Gamma)$ lies in the fundamental domain.
\end{lemma}
\begin{proof}
We have \begin{align*}G(GM\otimes^L_{R Q} \Gamma)
&=\Hom_{\mathcal{C}_{R Q}} (\Gamma, GM\otimes^L_{R Q} \Gamma) \\ &
\cong \Hom_{\mathcal{C}^{orb}} (R Q, G M) \\ &\cong \Hom_{R Q} (R Q,
G M)\cong GM, \end{align*} as the embedding of $\mathcal{C}^{orb}$
into $\mathcal{C}_{R Q}$ is fully faithful by \ref{faithful}. Let
now $G M$ be a lattice. Then $GM\otimes^L_{RQ} \Gamma$ is in $D_{\le
0}$. We have $\Hom_{\Gamma}(GM\otimes^L_{RQ}\Gamma,Y) \cong
\Hom_{\per (RQ)} ( G M, R\Hom_{ \Gamma} (\Gamma, Y))$. If $Y \in
D_{\le -2}$, then $R\Hom_{ \Gamma} (\Gamma, Y)$ also lies in $D_{\le
-2}$. As $G M$ is a lattice, we have that $GM $ lies in ${ }^\perp\!
D_{\le -2}$ and hence $\Hom_{\per (RQ)} ( G M, R\Hom_{ \Gamma}
(\Gamma, Y))$ vanishes. This finishes the proof.
\end{proof}

\begin{defi}
We call an indecomposable object $X$ in $\mathcal{C}_{R Q}$
{\em lattice-like}, if there is a lattice $L$ such that $X$ is
isomorphic to $\Gamma\otimes^L_{R Q} L$ or $X$ is isomorphic to
$\Sigma \Gamma \otimes^L_{R Q} P $ for a projective indecomposable
$R Q$-module $P$.
\end{defi}
All lattice-like objects are images of objects in the orbit category $\mathcal{C}^{orb}$.

\begin{theo}\label{rigid}
Let $M$ be an indecomposable rigid object of $\mathcal{C}_{R Q}$.
There is an $N\in \Z$ such that $M$ is isomorphic to $G(\Sigma^N
M)\otimes^L_{R Q} \Sigma^{-N} \Gamma$.
\end{theo}

\begin{proof}
By \ref{fundamental}, there is an $N\in \Z$ and an object $M' \in
\mathcal{F}[N]$ such that $\pi(M')=M$. We assume without loss of
generality that $N=0$. By Lemma~\ref{add}, each object of
$\mathcal{F}$ admits an $\add(\Gamma)$-resolution. Therefore, for
$N'$ in $\mathcal{F}$, we have, as in Proposition 2.1 c) of
\cite{KellerReiten08}, the isomorphism
\[
\Hom_{\mathcal{C}_{R Q}/\add (\Sigma \Gamma)} (\pi M', \pi N') \cong
\Hom_{R Q} ( G M, G \pi N').
\]
Note also that $G(GM\otimes^L_{RQ} \Gamma)$ is isomorphic to $G M$
by Lemma~\ref{G}. Let
$$\Sigma^{-1} M\to  P_1 \stackrel{h} \to P_0 \to M $$ be an
$\add(\Gamma)$-resolution in the integral cluster category. Then all
morphisms from $ P_1$ to $M$ factor through $h$. Applying $G$ to the
triangle gives the start of a projective resolution $$GP_1 \to GP_0
\to GM \to 0.$$ As $P_1$, $P_0$ and $M$ are all images of objects in
$\mathcal{F}$, we have that every morphism from $GP_1$ to $GM$
factors through $Gh$. Therefore $GM$ is rigid as an $R Q$-module and
hence is a lattice. If $GM$ vanishes, then $M$ lies in
$\add(\Sigma\Gamma)$. Since we have equivalences
\[
\add(RQ) \iso \add(\Gamma) \iso \add(\pi(\Gamma)) \iso \add(RQ) \ko
\]
we obtain $M \cong G(\Sigma^{-1}M)\ten_{RQ}^L\Sigma \Gamma  $. So
let us suppose that $GM$ does not vanish. As $G M$ is a lattice, the
object $GM\otimes^L_{R Q} \Gamma$ lies in $\mathcal{F}$. It follows
that there are isomorphisms
\[
f \in \Hom_{\mathcal{C}_{R Q}/\add (\Sigma \Gamma)} ( M,
GM\otimes^L_{R Q} \Gamma) \mbox{  and } g  \in \Hom_ {\mathcal{C}_{R
Q}/\add (\Sigma \Gamma)} ( GM\otimes^L_{R Q} \Gamma, M).
\]
We lift $f$ and $g$ to morphisms $\tilde{f}$ and $\tilde{g}$ in the integral cluster
category. Then $\tilde{f} \tilde{g}$ lies in
\[
\Hom_{\mathcal{C}_{R Q}}( GM\otimes^L_{R Q} \Gamma,GM\otimes^L_{R Q}
\Gamma) \cong \Hom_{\mathcal{C}^{orb}}(G M, GM).
\]
Now the functor $\Hom_{\cc^{orb}}(RQ,?)$ induces a surjective
ring homomorphism
\[
\Hom_{\cc^{orb}}(GM,GM) \to \Hom_{RQ}(GM,GM)
\]
whose kernel is a radical ideal. Since $fg$ is an isomorphism of
$RQ$-modules, $\tilde{f}\tilde{g}$ is an isomorphism in the integral
cluster category. But $M$ is indecomposable and $\tilde{f}\tilde{g}$
factors through $M$, hence the objects  $M$  and $ GM\otimes^L_{R Q}
\Gamma$ are isomorphic.
\end{proof}
Note that the proof of the previous theorem also holds if
$G(\Sigma^N M)$ is a non vanishing lattice. Therefore we have
\begin{lemma}
Let $M$ be an indecomposable object in $\mathcal{C}_{RQ}$ such that
there is a $Z\in \mathcal{F}$ with $\pi(Z) $ is isomorphic to $M$.
If $GM$ is a non vanishing lattice, then $M$ is isomorphic to
$GM\otimes^L_{RQ}\Gamma$.
\end{lemma}
The next result is well-known for derived categories of hereditary algebras over fields.
\begin{lemma}\label{S} Let $R$ be a principal ideal domain.
The Serre functor $S_R$ of $\per (RQ)$ maps shifts of rigid lattices
to shifts of rigid lattices.
\end{lemma}
\begin{proof} The Serre functor is given by
the left derived functor of tensoring with the bimodule $\Theta=
\Hom_{R} (R Q, R)$. As $\F Q$ is hereditary, the Serre functor
$S_{\F}$ maps a non-projective module $L$ to $\Sigma\tau L$, where
$\tau$ denotes the Auslander-Reiten translation of the category of
$\F Q$-modules. Hence the Serre functor applied to non projective
indecomposable $\F Q$-modules has non-vanishing homology only in
degree minus one. Moreover, the functor $S_{\F}$ maps projective
modules to injective modules.

The statement is clear for projective lattices of $RQ$. Let $M$ be a
non-projective indecomposable rigid lattice over $RQ$ and
\[
0 \to P_1 \stackrel{f} \to P_0 \to M \to 0
\]
a projective resolution of
$M$. Then $P_0$ and $P_1$ are lattices and $f$ splits as a map of
$R$-modules. The object $S_R(M)$ is isomorphic to the complex
$$\cdots \to 0 \to P_1 \otimes_{RQ} \Theta \stackrel{f \otimes
\Theta}\to P_0 \otimes_{RQ} \Theta \to 0 \to \cdots$$ As $\Theta$ is
a lattice, so are $P_0\otimes_{RQ} \Theta$ and $P_1 \otimes_{RQ}
\Theta$. The cokernel of $f \otimes \Theta$ is given by $M
\otimes_{RQ} \Theta$. We show next that $f\otimes \Theta$ is surjective
by proving that $M\otimes_{RQ} \Theta$ vanishes. By the proof of
\ref{scalar extension}, $S_R$ commutes with the functor
$-\otimes^L_R \F$. Therefore $(\F\otimes_R  \Theta)\otimes_{ \F Q}
(M\otimes_R \F )$ and $ (\Theta\otimes_{RQ} M)\otimes_{R} \F$ are
isomorphic. By \cite[Theorem 2]{CrawleyBoevey96} the module
$M\otimes_R \F$ is a rigid indecomposable non-vanishing lattice
which is non-projective. Suppose that $M\otimes_{RQ} \Theta$ does
not vanish. Then there is a field $\F$ such that
$(\Theta\otimes_{RQ} M)\otimes_{R} \F$ does not vanish. Hence
$S_{\F}( \F \otimes_R M) $ has non-vanishing cohomology in degree
zero, which is a contradiction, as $\F\otimes_R M$ is a
non-projective $\F Q$-module.  Therefore $f\otimes \Theta$ is
surjective. Its kernel has to be a lattice as it is a submodule of a
lattice. The object $S_R(M)$ is isomorphic to the one-shift of this
lattice.
\end{proof}

Next we analyze the relationship between the integral cluster
category and the cluster category over the field $\F$. We can strengthen \cite[A.8]{Keller10b}.
\begin{theo}\label{rigid1} Let $R$ be a principal ideal domain.
(1) The rigid indecomposable objects of $\cc_{RQ}$ are lattice-like.

(2) The reduction functor $$?\otimes_{R} \F: \mathcal{C}_{R Q} \to
\mathcal{C}_{\F Q}$$ induces a bijection from the set of isomorphism
classes of rigid indecomposable objects in $\mathcal{C}_{R Q} $ to
the set of isomorphism classes of rigid indecomposable objects of
$\mathcal{C}_{\F Q}$.
\end{theo}
\begin{proof}
We denote by $F_R$ the functor $S_R\Sigma^{-2}$ in $\per (RQ)$ and by
$F$ the functor $S\Sigma^{-2}$ in $\per (\F Q)$. Let $M \in
\mathcal{C}_{R Q}$ be a rigid indecomposable object. By \ref{rigid}
and \ref{faithful}, we can view $M$ as an object of
$\mathcal{C}^{orb}$ and $M$ is isomorphic to the $N$-shift  of the
image of a rigid $R Q$-lattice $M'$. Then $M\otimes^L_{R} \F$ is
isomorphic to the $N$-shift of the rigid module $M'\otimes_{R} \F$
seen as an object in $ \mathcal{C}_{\F Q}$.  If we view $\Sigma^{N}
M' \otimes^L_{R} \F$ as an object in $\per (\F Q)$, we see that there
is a rigid indecomposable module $L$ in $\mod \F Q$ or  an
indecomposable direct summand $P$ of $ \F Q$ such that $\Sigma^{N}
M'\otimes_{R} \F$ and $L$ lie in the same $F$-orbit or $\Sigma^{N}
M' \otimes^L_{R} \F$ and $\Sigma P$ lie in the same $F$-orbit.  Let
$n\in \Z$ be such that $ L \cong F^n\Sigma^{-N} M'\otimes_{R} \F$ or
$ \Sigma P \cong F^n \Sigma^{-N} M'\otimes_{R} \F$.

As $S_{R}$ maps the shift of a rigid lattice to the shift of a rigid
lattice by \ref{S}, we have that $F_R^n \Sigma^N M' $ is also the
$k$-shift of a rigid $R Q$-lattice, say $L'$ in $\per (RQ)$ for some
$k\in \Z$. By \ref{scalar extension}, we have that $\Sigma^ k L'
\otimes \F$ and $L$ are isomorphic or $\Sigma^k L' \otimes_R F \cong
\Sigma P$ in $\per (\F Q)$, hence $k$ vanishes in the first case and
$k$ equals one in the second case. Furthermore, in the second case
$L'$ is isomorphic to a projective $R Q$-module. We obtain therefore
that $\Sigma^{N} M'$ is in the $F_R $-orbit of $L'$ in the first
case and is in the $F_R$-orbit of $\Sigma L'$ in the second case.
Hence in the orbit category, we have that $M$ is isomorphic to a
lattice or to the one-shift of a projective lattice. This finishes
the proof of the first statement. Using Theorem~1 of
\cite{CrawleyBoevey96} we then immediately obtain the second
statement.
\end{proof}
Note also that all rigid objects satisfy the unique decomposition property,
as they are lattice-like and the statement holds by
\cite[Theorem 2]{CrawleyBoevey96} for rigid lattices in the category of $RQ$-modules.

We can also show that the orbit category $\mathcal{C}^{orb}$ and the
integral cluster category coincide and hence the orbit category is
triangulated.

\begin{theo}\label{triangulated}
The embedding of the orbit category $\mathcal{C}^{orb}$ into
$\mathcal{C}_{RQ}$ is an equivalence. Therefore the orbit category
$\mathcal{C}^{orb}$ is canonically triangulated.
\end{theo}
\begin{proof}
We consider the commutative diagram of functors
\[
\xymatrix{ \per (RQ) \ar[r]^-{?\otimes^L_{RQ}  \Gamma} \ar[d] & \per (
\Gamma) \ar[d]_{\pi}
\\ \mathcal{C}^{orb}  \ar[r]_{? \otimes^L_{RQ}  \Gamma} & \mathcal{C}_{RQ} .}
\]
By \ref{embedding}, the bottom functor is fully faithful. Let us
show that it is essentially surjective. Let $M \in \cc_{RQ}$.
By~\ref{fundamental}, there is an $n \in \Z$ and an $M' \in
\mathcal{F}$ such that $\Sigma^n \pi M' \cong M$. We assume without
loss of generality that $n=0$ and chose an $\add(\Gamma)$-resolution
$P_1 \stackrel{h} \to P_0 \to M' \to \Sigma P_1.$ By
remark~\ref{c-orbit-morphisms}, the restriction of $\pi$ to
$\add(\Gamma)$ is fully faithful and so is the restriction of
$\per(RQ) \to \per(\Gamma)$ to $\add(RQ)$. Thus, the morphism $h:P_1
\to P_0$ is the image of a morphism in $\per (RQ)$. Since
$-\otimes^L_{RQ} \Gamma: \per (RQ) \to \per (\Gamma)$ is a triangle
functor, $M'$ is also isomorphic to an image of an object in $\per
RQ$. By the commutativity of the diagram, we deduce that $M$, as an
object in $\mathcal{C}_{R Q}$, is isomorphic to the image of an
object in $\mathcal{C}^{orb}$. Now the objects in $\mathcal{
C}_{RQ}$ are identical with the objects in $\per (\Gamma)$, hence
$?\otimes^L_{RQ} \Gamma: \mathcal{C}^{orb}\to \mathcal{C}_{R Q}$ is
essentially surjective and hence an equivalence.
\end{proof}

\begin{cor}\label{2cy}
 The integral cluster category satisfies the relative
2-Calabi-Yau property, i.e. $X$ and $Y \in \mathcal{C}_{RQ}$, there
is a bifunctorial isomorphism
$$ \rhom_{R}(\rhom_{\mathcal{C}_{RQ}}( X,Y),R)\cong \rhom_{\mathcal{C}_{RQ}}(Y, \Sigma^2 X)$$
in $D(R)$.
\end{cor}

\section{Cluster-tilting mutation}
Mutations of cluster-tilting objects have been defined for cluster
categories over fields in \cite{BuanMarshReinekeReitenTodorov06},
generalizing the mutations of tilting objects in hereditary
categories studied in \cite{HappelUnger05a}. The mutation of
rigid objects and cluster-tilting objects in the cluster category
is used in \cite{CalderoKeller06} to give an additive
categorification of the cluster algebra associated to the quiver
$Q$ and its exchange relations. We refer to \cite{BuanMarsh06}
 \cite{Keller09b} \cite{Keller10b} \cite{Reiten06} for overviews.

Using our classification of rigid objects in the integral cluster
category, we can generalize the results obtained in
\cite{BuanMarshReinekeReitenTodorov06}. Throughout this section,
we assume that $R$ is a principal ideal domain and we
fix a ring homomorphism from $R$ to a field $\F$. Let $Q$ be a
finite quiver without oriented cycles and let $n$ be the number of
its vertices.
\begin{defi}
A {\em cluster-tilting object} $T$ is a rigid object in $\cc_{RQ}$
such that $T$ has $n$ indecomposable direct summands which are
pairwise non-isomorphic. Let $T'$ be another cluster-tilting
object. The pair $(T,T')$ is called a {\em mutation pair} if $T$
and $T'$ have exactly $n-1$ isomorphic indecomposable summands
in common.
Then we say that $T'$ is connected to $T$ by a {\em cluster-tilting
mutation.}
\end{defi}
By Theorem \ref{rigid1} the results of
\cite{BuanMarshReinekeReitenTodorov06}, every rigid
indecomposable object appears as a direct summand of a
cluster-tilting object. Moreover, the functor $? \otimes_R \F$
induces a bijection from the set of isomorphism classes of
cluster-tilting objects of $\mathcal{C}_{RQ}$ onto that of
$\cc_{\F Q}$ and this bijection preserves mutation pairs.

\begin{lemma}\label{ext}
If $X$ and $Y$ are rigid objects in $\mathcal{C}_{RQ}$, then
$\Ext_{\cc_{RQ}}^1(X,Y)$ is a free $R$-module and $? \otimes_R \F$
induces a bijection between $ \Ext_{\cc_{RQ}}^1(X,Y)$ and $
\Ext_{\cc_{\F Q}}^1( \F \otimes_R X,\F \otimes_R Y)$.
\end{lemma}
\begin{proof}
Let first $X$ and $Y$ be two rigid $RQ$ lattices. By \cite[Theorem
1]{CrawleyBoevey96}, the $R$-module $\Ext_{RQ}^1(X,Y)$ is free.
By applying \ref{2cy} we obtain that the $R$-module
$\Ext^1_{\mathcal{C}_{R Q}}(X, Y)$ is isomorphic to 
\[
\Ext^1_{RQ}(X, Y) \oplus \Hom_R(\Ext^1_{RQ}(Y,X), R),
\]
and hence is free.
If we apply $\F \otimes_R ?$, we obtain, again
by \cite[Theorem 1]{CrawleyBoevey96}, that it is isomorphic to
\[
\Ext^1_{\F Q}(\F \otimes_R X, \F \otimes_R Y) \oplus 
\Hom_{\F Q}(\Ext^1_{\F Q}(\F \otimes_R Y,\F \otimes_R X), \F),
\] 
which is isomorphic to $\Ext^1_{\cc_{\F Q}}(X,Y)$. If $Y \cong \Sigma P$
for some projective $RQ$-module $P$, then $\Ext^1_{\cc_{RQ}}(X
\ten_R \F,Y\ten_R \F)$ is isomorphic to $\Hom_{RQ} (X,P)$ which is
also a free $R$-module by \cite[Theorem 1]{CrawleyBoevey96}. The
rest of the proof is analogous.
\end{proof}
\begin{theo}[Cluster tilting mutation]
Let $T$ be a cluster tilting object of $\cc_{RQ}$ and $X$ an indecomposable
direct summand of $T$ with complement $X'$. Let $Y$ be an indecomposable rigid object. Then $T':=Y\oplus X'$ is a cluster-tilting object if and only if  
$\Ext^1_{\mathcal{C}_{R Q}}(X, Y)$ has rank one.
\end{theo}
\begin{proof}
By \ref{ext} we have $\Ext^1_{\cc_{RQ}} (X, Y) \otimes_R \F \cong
\Ext^1_{\cc_{\F Q}}(X \otimes_R \F, Y \otimes_R \F)$. Furthermore
both objects $\F \otimes_R X$ and $\F  \otimes_R Y$ are rigid and
indecomposable. Clearly $\F \otimes_R T \cong \F \otimes_R X\oplus
\F \otimes X'$ is a cluster tilting object in $\cc_{\F Q}$. Thus,
by \cite[7.5]{BuanMarshReinekeReitenTodorov06}, the object $\F
\otimes_R T'$ is cluster tilting if and only if the extension
group $\Ext^1_{\cc_{\F Q}}(X \otimes_R \F, Y \otimes_R \F)$ is one
dimensional. As the functor $\F \otimes_R ?$ induces a bijection
between rigid indecomposable objects in $\cc_{RQ}$ and $\cc_{\F
Q}$, the object $T'$ is cluster-tilting if and only if 
$\Ext^1_{\mathcal{C}_{R Q}}(X, Y)$ has rank one.
\end{proof}
Let $X$ and $Y $ be rigid indecomposable objects with an extension
space of rank one. By the preceding theorem and the results of
\cite{BuanMarshReinekeReitenTodorov06}, we obtain that
there is a rigid object $X'$ such that
$Y\oplus X'$ and $X\oplus X'$ are cluster-tilting objects. 
Let us choose generators $\eps$ and $\eps'$ of the 
rank one modules $\Ext^1_{\mathcal{C}_{R Q}}(X, Y)$ and
$\Ext^1_{\mathcal{C}_{R Q}}(Y,X)$. We construct
non split triangles
\[ 
Y\stackrel{f} \to E \to X \stackrel{\eps} \to \Sigma Y \mbox{   and  } 
X \to E' \stackrel{g} \to Y \stackrel{\eps'} \to \Sigma X.
\]
By Lemma
\ref{ext} these triangles are mapped by the functor $\F\otimes_R ?
$ to non-split triangles in $\mathcal{C}_{\F Q}$. By
\cite[6.4]{BuanMarshReinekeReitenTodorov06} the maps $\F\otimes_R
f$ and $\F\otimes_R g$ are minimal $\add (\F \otimes_R
X')$-approximations. 
We call the triangles in the integral cluster
category the {\em exchange triangles} of the mutation. By
\cite{CalderoKeller06} they categorify the exchange relations in
the cluster algebra associated to the quiver $Q$.

It was shown in \cite{BuanMarshReinekeReitenTodorov06}
\cite{HappelUnger05a}, cf. also \cite{Hubery10}, that all cluster-tilting objects 
of $\cc_{\F Q}$ are related by iterated mutation. Clearly, as cluster-tilting objects and their mutations in $\cc_{RQ}$ are in bijection with cluster-tilting objects in $\cc_{\F Q}$ and their mutations, we obtain the following result.
\begin{cor}
The cluster tilting objects in $\mathcal{C}_{R Q}$ are all
connected via cluster-tilting mutation and can therefore be obtained by iterated
mutation from the initial object~$\Gamma$.
\end{cor}

\def\cprime{$'$} \def\cprime{$'$}
\providecommand{\bysame}{\leavevmode\hbox to3em{\hrulefill}\thinspace}
\providecommand{\MR}{\relax\ifhmode\unskip\space\fi MR }
\providecommand{\MRhref}[2]{%
  \href{http://www.ams.org/mathscinet-getitem?mr=#1}{#2}
}
\providecommand{\href}[2]{#2}

\end{document}